\documentclass[letterpaper, 12pt]{article}
\usepackage[letterpaper, margin=1 in]{geometry}
\pdfoutput=1
\usepackage{amsmath, amsthm, amssymb}
\usepackage{bbold}
\usepackage[mathscr]{euscript}
\usepackage{graphicx}
\usepackage{epstopdf}
\usepackage{url}
\usepackage[pdftex=true, bookmarks=true, bookmarksnumbered=true, backref=false, colorlinks=false, linkbordercolor={1 .75 .75},citebordercolor={.75 1 .75}, linktocpage=false]{hyperref}
\hypersetup{pdfauthor={Eric Hart and B\'alint Vir\'ag},pdftitle={H\"older continuity of the integrated density of states in the one-dimensional Anderson model}}

\usepackage[longnamesfirst]{natbib}
\bibpunct[ ]{(}{)}{,}{a}{}{,}

\usepackage{setspace}
\usepackage{tocloft}

\usepackage{enumerate}

\allowdisplaybreaks[3]

\DeclareMathOperator{\arccosh}{arccosh}

\theoremstyle{plain}
\newtheorem{theorem}{Theorem}
\newtheorem{lemma}[theorem]{Lemma}
\newtheorem{corollary}[theorem]{Corollary}

\theoremstyle{definition}

\theoremstyle{remark}
\newtheorem{remark}[theorem]{Remark}

\title{H\"older Continuity of the Integrated Density of States in the One-Dimensional Anderson Model}
\author{Eric Hart\and B\'alint Vir\'ag}
\date{\today}
\begin{document}

\maketitle

\begin{abstract}
\noindent 
We consider the one-dimensional random Schr\"odinger operator 
\begin{equation*}
 H_{\omega} = H_0 + \sigma V_{\omega},
\end{equation*}
where the potential  $V$ has i.i.d. entries with bounded support. We prove that the IDS is H\"older continuous with exponent $1-c \sigma$.
This improves upon the work of Bourgain showing that the H\"older exponent tends to $1$ as sigma tends to $0$ in the more specific Anderson-Bernoulli setting.

\end{abstract}

\bigskip\bigskip\bigskip

\section{Introduction}
\subsection{The Anderson Model}
We consider the Anderson model for Random Schr\"odinger operators

\begin{equation}
H_\omega = H_0 + \sigma V_\omega
\end{equation}
where $H_0$ is the discrete Laplacian operator on $\ell^2\left(\mathbb{Z^d}\right)$, $V_\omega$ is a random potential (diagonal) operator, with iid random variables on the diagonal, and $\sigma$ is the coupling constant, a parameter regulating the amount of randomness in the model, so that taking $\sigma$ to be very small decreases the randomness. We will be working with the $1$-dimensional model (i.e. the model on $\ell^2\left(\mathbb{Z}\right)$), which can be expressed in matrix form as

\begin{equation*}
 H_\omega = \begin{bmatrix}
  \ddots \\ & 0 & 1 & 0 & 0 \\ & 1 & 0 & 1 & 0 \\ & 0 & 1 & 0 & 1 \\ & 0 & 0 & 1 & 0 \\ & & & & & \ddots
 \end{bmatrix}
 + \sigma
 \begin{bmatrix}
  \ddots \\ & v_{-1} & 0 & 0 & 0 \\ & 0 & v_0 & 0 & 0 \\ & 0 & 0 & v_1 & 0 \\ & 0 & 0 & 0 & v_2 \\ & & & & & \ddots
 \end{bmatrix}
\end{equation*}
where the $v_i$, referred to as single-site potentials, are iid random variables with common distribution $\mathbb{P}$.

\subsection{The Result}
Let $\mu_\sigma$ be the integrated density of states measure (IDS) for $H_\omega$. We have the following theorem:

\begin{theorem}
\label{mainTheorem}
  Consider the Anderson model under the conditions that $\mathbb{P}$ has mean $0$, variance $1$ and support bounded by $c_0$. For all $\gamma > 0$ the IDS, $\mu_\sigma$, restricted to the interval $(-2+\gamma,-\gamma)\cup(\gamma,2-\gamma)$, is H\"older continuous with exponent $1-460c_0^3\sigma/\gamma$. More precisely, for $\lambda_0 \in (-2+\gamma,-\gamma)\cup(\gamma,2-\gamma)$, $\sigma \leq 1$ and $\lambda \leq 1$

  \begin{equation*}
  \mu_\sigma [\lambda_0, \lambda_0+\lambda]  \leq \frac{2}{\sigma^3} \lambda^{1 - 460c_0^3\sigma / \gamma}.
  \end{equation*}
\end{theorem}

\subsection{Why the Anderson Model}
The Anderson model is used to consider a quantum mechanical particle moving through a disordered solid, feeling potential from atoms at the lattice sites, where the randomness of the potential corresponds to impurities in the solid; see, for example, the discussion in \cite{Inv}. The particle moving in $d$-dimensional space is given by a function $\psi$, and it's evolution by $e^{-itH_{\omega}} \psi_0$. With this view, the operator prescribes the time evolution of the particle, and properties of the spectrum of $H_\omega$, $\Sigma\left(H_{\omega}\right)$, correspond to questions about how electrons move through the wire. A natural question to ask is whether the generalized eigenfunctions are localized or delocalized, which can be thought of as a question about the conductive properties of the solid. When $\sigma = 0$ we imagine a metal with no impurities, which we expect to be a conductor. Indeed, the operator $H_0$ has spectrum $\left(-2,2\right)$, and its generalized eigenfunctions are not in $\ell^2$. On the other hand, in $1$-dimension, for any $\sigma >0$ one can show that the eigenfunctions become exponentially localized, a phenomenon known as Anderson localization. See for example the results of \cite{GMP}, \cite{KS}, and \cite{CKM}, the latter covering the case of Bernoulli-potentials.

\subsection{The Integrated Density of States}

The integrated density of states (IDS) can be thought of as the average number of eigenfunctions per unit volume in the spectrum. It can be obtained by restricting the operator to a finite box, and then taking the limit of the empirical eigenvalue distribution, see \cite{Inv}. Understanding the IDS is a first step in the study of the spectral properties of the random operator. When $\mathbb{P}$ is absolutely continuous, much is understood about the IDS. The main tool mathematicians use in this case is the celebrated estimate of \cite{Wegner}.  It bounds the expected number of eigenvalues in a small interval of the spectrum of a Schr\"odinger operator restricted to a finite box. This bound depends on the infinity norm of the density, and so only exists in the case where the distribution of the noise is absolutely continuous. The lack of this tool in cases where the noise is not absolutely continuous results in a bigger challenge to prove many expected results; even in the simple case where the noise has a Bernoulli distribution, referred to as the Anderson-Bernoulli model, much less is known.

It is natural to ask further questions about the IDS, such as what kind of continuity properties it has, and whether we can describe it more explicitly. One would expect that the IDS should be H\"older continuous for small coupling constants, and that the exponent should improve, specifically approach $1$ as $\sigma \downarrow 0$, see \cite{B1}.  This and more has been known when the noise is absolutely continuous for some time.  For example, Minami estimates -- bounds on the probability of seeing two eigenvalues in a small interval of the spectrum of a Schr\"odinger operator -- are even more refined than the Wegner estimate, can be proved in the continuous case, and are used in \cite{minami1996} to establish Poisson statistics of the spectrum.  On the other hand, when the noise is not absolutely continuous, it is possible for H\"older continuity to fail if $\sigma$ is not small enough. For example, \cite{ST} formalize a result of \cite{Halp} to show that, when the noise is Bernoulli, for any $\sigma > 0$, the IDS cannot be H\"older continuous with exponent greater than

\begin{equation*}
2 \log 2 / \arccosh \left(1+\sigma\right).
\end{equation*}
Since the maximum exponent of H\"older continuity is $1$ anyway, this result has no content for small sigma. On the other hand, for any $\sigma > 9/8 = \cosh\left( 2\log2 \right)-1$, the exponent of H\"older continuity must be bounded away from $1$.

\subsection{H\"older Continuity}
In \cite{SVW} H\"older continuity is established in the Anderson-Bernoulli model for certain coupling constants, but the exponent in that paper gets worse instead of better as $\sigma$ decreases.  \cite{B1} establishs that the H\"older continuity doesn't break down as $\sigma$ decreases, and the exponent must tend to at least $1/5$.  This result is improved in \cite{B2}, where he gives a non-quantitative bound to show that the Holder exponent converges to $1$ as $\sigma \downarrow 0$. Following his argument carefully it seems that his methods yield a boud of the form
\begin{equation*}
1-c|\log\left(\sigma\right)|^{-1/2}.
\end{equation*}
In contrast, our result gives that the speed with which the exponent tends to $1$ is bounded by

\begin{equation*}
 1-c\sigma
\end{equation*}
where our value of $c$ is explicit. In both our result and Bourgain's the constant depends on the energy being considered, in particular it gets large at energies near the edge of the spectrum, but also near $0$. However, our method applies to a wider class of noise distributions than Bernoulli, specifically our main assumption is that $\mathbb{P}$ has finite support. Our assumptions that $\mathbb{P}$ has mean $0$ and variance $1$ are for ease of notation.

The breakdown of this work is as follows. In Section \ref{sect2} we use the method of Transfer matrices to view the eigenvalue equation for the finite-level Schr\"odinger operator as a product of $2 \times 2$ matrices, and get some geometric intuition by viewing this matrix product as a random walk in the (upper half) complex plane via projectivization. In Section \ref{sect3} we prove a deterministic result (Theorem \ref{tmbktkblp}) relating the number of eigenvalues in a small interval of the finite-level Schr\"odinger operator to the number of large backtracks of the imaginary part of a random walk (with drift) defined in Section \ref{sect2}. We also bound the jumps of the real part of this random walk. In Section \ref{sect4} we use the known Figotin-Pastur recursion, most clearly laid out in \cite{BS2}, and a Martingale argument to bound the probability of large backtracks of random walks like the one in Section \ref{sect2} (Theorem \ref{tmlowprob}). Finally, in Section \ref{sect5} we carefully choose some parameters and apply the results of Sections \ref{sect2} and \ref{sect3} to bound the probability of the number of eigenvalues in a small interval of the finite-level Schr\"odinger operator, and take a limit to obtain the main result.

\section{Preliminaries}
\label{sect2}

\subsection{The Transfer Matrix Approach}

Consider the $1$-dimensional random Schr\"odinger operator in the Anderson model $H_\omega = H_0 + \sigma V_\omega$.  We will be working with the restriction of this operator to a finite box, $H_{\omega, n}$. Since $H_\omega$ is tri-diagonal, the eigenvalue equation

\begin{equation*}
 H_{\omega,n} \phi = \lambda \phi
\end{equation*}
can be solved recursively in order to determine if a given $\lambda$ is an eigenvalue. Doing so allows us to write down an equivalent formulation of the eigenvalue equation:

\begin{equation}
\label{evOrig}
\begin{bmatrix} \phi_{n+1} \\ \phi_n \end{bmatrix} =  T^{\left(\lambda\right)}_{n} T^{\left(\lambda\right)}_{n-1} \cdots T^{\left(\lambda\right)}_{1} \begin{bmatrix} \phi_{1} \\ \phi_0 \end{bmatrix}
\end{equation}
where we set $\phi_{n+1} = \phi_0 = 0$ and the $T$ matrices are given by

\begin{equation*}
T^{\left(\lambda\right)}_i = \begin{bmatrix} \lambda - \sigma\omega_i & -1 \\ 1 & 0 \end{bmatrix}.
\end{equation*}

Note that $\phi_n$ in this equation is unknown, and that by linearity we may let $\phi_1 =1$, which is allowed because $\phi_1$ can't be $0$, since if it were, the recursion would imply that $\phi \equiv 0$. This rewriting of the eigenvalue equation is a common technique when studying the spectrum of Schr\"odinger operators in the Anderson model, often called the transfer matrix approach. One immediate benefit of this approach is that we can use the transfer matrices to define the Lyapunov exponent, $\gamma_{\sigma}\left(\lambda\right)$, a quantity which captures the speed at which the product of these transfer matrices grows, as follows

\begin{equation*}
 \gamma_{\sigma} \left(\lambda\right) = \lim_{n \to \infty} \frac{1}{n} \log ||T_i^{\left(\lambda\right)}||.
\end{equation*}
The Lyapunov exponents of Schr\"odinger operators can give us information about the operators themselves. For example, the authors in \cite{CL} give a theorem excluding H\"older continuity of the IDS for operators with large Lyapunov exponents.

\subsection{The Complex Plane}
To help with intuition, we will identify the objects we're working with in the upper half of the complex plane (UHP). Specifically, we can view the transfer matrices $T_i^{\left(\lambda\right)}$ as automorphisms of the UHP through projectivization. Given some (complex) $2$-vector

\begin{equation*}
 v = \begin{bmatrix} v_1 \\ v_2 \end{bmatrix}
\end{equation*}
we think of its projectivization as the point 

\begin{equation*}
\mathscr{P} \left[v\right] = \frac{v_1}{v_2}
\end{equation*}
in the complex plane. Then a $2 \times 2$ matrix 

\begin{equation*}
M =  \begin{bmatrix} a & b \\ c & d \end{bmatrix}
\end{equation*}
can be thought of as an automorphism of the plane as

\begin{equation*}
 M \circ v = \mathscr{P} \left[M \begin{bmatrix} v_1 \\ v_2 \end{bmatrix}\right] = \frac{a \mathscr{P}\left[v\right] + b}{c \mathscr{P}\left[v\right] + d}.
\end{equation*}
While the UHP will be the most useful model for us to think about our objects geometrically, occasionally things will be easier to understand in the context of the disk. For example, a certain automorphism of the half plane may be most easily understood as a ``rotation'' if it corresponds to mapping the UHP to the disk with a Cayley transform, applying a rotation to the disk, and then mapping the result back to the UHP. In such cases, we may call such an automorphism a rotation for simplicity.

\subsection{More on Transfer Matrices} \label{MoTM}
We will be investigating the spectrum by fixing a particular point, or energy in the spectrum, $\lambda_0$, and looking at the spectrum near this energy. For a fixed $\lambda_0$, define $\theta$, $\rho$, and $z$ by

\begin{equation*}
 \lambda_0 =: 2\cos \theta, \hspace{5pt} 0 \leq \theta \leq \pi
\end{equation*}

\begin{equation*}
 \rho := \frac{1}{\sqrt{4-\lambda_0^2}} = \frac{1}{2 \sin \theta}
\end{equation*}
and

\begin{equation*}
 z := \left(\lambda_0 + i/\rho\right)/2 = e^{i\theta}. 
\end{equation*}
To simplify notation we suppress the $\lambda_0$ when it appears in the transfer matrices, writing

\begin{equation*}
T^{\left(\lambda_0\right)}_i = T_i = \begin{bmatrix} \lambda_0 - \sigma\omega_i & -1 \\ 1 & 0 \end{bmatrix}.
\end{equation*}
Finding eigenvalues near $\lambda_0$ means solving equation (\ref{evOrig}) for $\lambda_0 + \lambda$.
If we define 

\begin{equation*}
 Q = \begin{bmatrix} 1 & 0 \\ -\lambda & 1 \end{bmatrix}
\end{equation*}
then $T_i^{\left(\lambda_0+\lambda\right)} = T_i Q$, and we can substitute this into equation (\ref{evOrig}), evaluated at $\lambda_0 + \lambda$, to get

\begin{equation*}
\begin{bmatrix} \phi_{n+1} \\ \phi_n \end{bmatrix} =  T_{n} Q T_{n-1} Q \cdots T_{1}Q \begin{bmatrix} \phi_{1} \\ \phi_0 \end{bmatrix} 
\end{equation*}
which we can rearrange to obtain

\begin{equation}
\label{longEqn}
 \left(T_1\right)^{-1} \left(T_2\right)^{-1} \cdots \left(T_n\right)^{-1} \begin{bmatrix} 0 \\ \phi_n \end{bmatrix}  = Q^{T_{n-1} T_{n-2} \cdots T_{1}} Q^{T_{n-2} T_{n-3} \cdots T_{1}} \cdots Q \begin{bmatrix} \phi_{1} \\ 0 \end{bmatrix}
\end{equation}
with the notation $Q^A$ being conjugation of $Q$ by $A$. This expression is convenient because all of the randomness on the right hand side is in the conjugation, but $\lambda$ only appears in $Q$, which has no randomness. This allows us to easily view the process as a random walk.
To simplify notation, let $W_{i} = T_{i} T_{i-1} \cdots T_{1}$, call the expression on the left hand side of (\ref{longEqn}) $v_*$, i.e.

\begin{equation*}
 v_* = W_n^{-1} \begin{bmatrix} 0 \\ \phi_n \end{bmatrix}
\end{equation*}
and let $V_n$ be the expression on the right hand side of equation (\ref{longEqn}) so that (by reversing the sides of the equation) we may rewrite (\ref{longEqn}) as

\begin{equation}
\label{defSuperman}
V_n := \begin{bmatrix} v_{1,n} \\ v_{2,n} \end{bmatrix} = Q^{W_{n-1}} Q^{W_{n-2}} \cdots Q^{W_1} Q \begin{bmatrix} \phi_{1} \\ 0 \end{bmatrix}  = v_*.
\end{equation}
The sequence $\{ W_k^{-1} \circ z \}_{k=1}^n$ defines a process in the UHP, and the sequence $\{\mathscr{P}\left[ V_k \right] \}_{k=1}^n$ defines a process on the boundary of the UHP plane.  Each $V_k$ is obtained by applying the automorphism $Q^{W_{k-1}}$ to the previous point, starting at the point at infinity, given by the projectivization of

\begin{equation*}
p = \begin{bmatrix} \phi_{1} \\ 0 \end{bmatrix}. 
\end{equation*}
Let $s_k$ be the projectivization of $V_k$, in other words

\begin{equation*}
s_k = \mathscr{P} \left[V_k\right] = v_{1,k}/v_{2,k}
\end{equation*}
and, keeping in mind that the process 

\begin{equation*}
 W_n^{-1} \begin{bmatrix} z \\ 1 \end{bmatrix}
\end{equation*}
corresponds to the process $W_n^{-1} \circ z$ in the UHP model, we will split this process up into its real and imaginary parts so that 

\begin{equation*}
X_n + iY_n := W_n^{-1} \circ z . 
\end{equation*}
With the understanding of the process $W_n^{-1} \circ z$ as a process in the UHP, and its separation into real and imaginary parts, we are able to state our main theorems.

\subsection{Main Theorems}

If $Y$ is a real valued process, then whenever $Y$ increases by $B$, we call this a \textit{backtrack} of $Y$ by an amount $B$. Note that this terminology makes more sense for processes with drift down. In particular it makes sense for the imaginary parts of random walks in the UHP which converge to the boundary.

\begin{theorem}
\label{tmbktkblp}
Let $\lambda_0 \in \left(-2,0\right) \cup \left(0,2\right)$, $n \in \mathbb{N}$, $\lambda>0$ and $\epsilon >0$. Fix $M$, let $0 < \beta \leq \left(2M\right)^{-1}$, and assume that $|\Delta X_k|/Y_k = |X_k - X_{k-1}|/Y_k \leq M$ for all $k \leq n$.  Then the number of eigenvalues of $H_{\omega,n}$ in the interval $\left[\lambda_0, \lambda_0 + \lambda\right]$ can be no more than $1$ plus the number of backtracks of the process $\log Y_n + \left[ \left(\epsilon + \lambda\beta\right)/\sin\theta + 2M\beta \right] n$ that are at least as large as $\log \left(\epsilon\beta/\lambda\right)$.
\end{theorem}

%

\begin{theorem}
\label{tmlowprob}
Assume $\sin 2\theta\not=0$. Let $E\left(\omega_j\right) = 0$, $E\left(\omega_j^2\right) = 1$, $|\omega_j| < c_0 $, and $\sigma \leq \frac{2\sin\theta |\sin 2\theta|}{460c_0^3}$. Also assume $\kappa \leq 6c_0^3 \rho^3 \sigma^3 / |\sin2\theta|$.
Then the probability that the process $\log Y_n + \kappa n$ has a backtrack of size $B$ starting from time $1$ is at most 

\begin{equation*}
2e^{-B (1 - 230c_0^3 \sigma / 2\sin\theta |\sin 2 \theta|)}.
\end{equation*}
\end{theorem}

\section{Random Schr\"odinger Operator and Random Walks}
\label{sect3}

\subsection{Walk on the Boundary of the UHP}
The process $V_k$ can be viewed as a random walk on the boundary of the UHP via projectivization. Since 

\begin{equation*}
Q \circ v = \frac{v}{1-\lambda v} 
\end{equation*}
there is reason to think of the matrix $Q$ as moving points $v$ on the boundary of the UHP ``to the right''. Since $\lambda$ is small, it certainly does this when $v$ is not too large. If $v$ is very large, it is possible that $Q \circ v < v$, but in this case we will think of $Q$ as having moved $v$ ``to the right, past $\infty$''. In this sense, conjugates of $Q$ also move points ``to the right'' along the boundary of the UHP.


With this in mind, we view the process $V_n$ as a random walk on the boundary of the UHP moving only to the right, so the notion of ``how many times this process passes a fixed point'' makes sense. On the other hand, since (\ref{defSuperman}) is just a rearrangement of the eigenvalue equation for the Schr\"odinger operator $H_{\omega,n}$, we make the following observation: for a fixed $n$ and $\lambda$ if

\begin{equation*}
 Q^{W_{n-1}} Q^{W_{n-2}} \dots Q \begin{bmatrix} \phi_{1} \\ 0 \end{bmatrix}  = v_*
\end{equation*}
then $\lambda_0 + \lambda$ is an eigenvalue of $H_{\omega,n}$. This motivates the following well known fact:

\begin{lemma}
\label{keyObv}
The number of eigenvalues of $H_{\omega,n}$ in the interval $\left[\lambda_0, \lambda_0 + \lambda\right]$ is equal to the number of times that the process $Q_{\lambda}^{W_{k-1}} Q_{\lambda}^{W_{k-2}} \dots Q_{\lambda} \left(p\right)$ passes the point $v_*$ as $k$ goes from $1$ to $n$.
\end{lemma}
Note: the idea here is that for a fixed $n$ we plan to count the eigenvalues of $H_{\omega,n}$ by considering each $Q^{W_k}$ as one step in a process, and looking at the behaviour of that process as $k$ goes from $1$ to $n$.

\begin{proof}
This proof from \cite{KMaV}. Let $B = \left[\lambda_0, \lambda_0 + \lambda\right] \times \left[0,n\right]$. By interpolating linearly to continuous time, we may consider the continuous map $f: B \to S^1$ given by

\begin{equation*}
 f\left(\lambda, t\right) =Q_{\lambda \left(t-1 - \lfloor t-1 \rfloor\right)}^{W_{\lceil t -1 \rceil }} Q_{\lambda}^{W_{\lfloor t-1 \rfloor}} Q_{\lambda}^{W_{\lfloor t-2 \rfloor}} \dots Q_{\lambda} \left(p\right).
\end{equation*}
Consider the loop given by going around the perimeter of $B$, i.e. from $\left(\lambda_0,0\right)$ to $\left(\lambda_0+\lambda,0\right)$ to $\left(\lambda_0+\lambda,n\right)$ to $\left(\lambda_0,n\right)$ and back to $\left(\lambda_0,0\right)$. Since $B$ is simply connected, the image of $f$ is topologically trivial. Further, $f\left(\left[\lambda_0, \lambda_0 + \lambda\right] \times \{0\}\right) = f\left(\{ \lambda_0\} \times \left[0,n\right]\right) = p$. Therefore, $f\left(\{ \lambda \} \times \left[0,n\right]\right)$ and $f\left(\left[\lambda_0,\lambda_0+\lambda\right] \times \{ n \}\right)$ must have opposite winding numbers. In other words, the number of times that the process 

\begin{equation*}
\{ V_k \}_{k=1}^n 
\end{equation*}
passes the point $v_*$ is equal to the number of times that the process 

\begin{equation*}
Q_{\lambda^*}^{W_{n-1}} Q_{\lambda^*}^{W_{n-2}} \dots Q_{\lambda^*} \left(p\right) 
\end{equation*}
passes the point $v_*$ as $\lambda^*$ is varied from $0$ to $\lambda$. By the observation above, the latter is clearly the number of eigenvalues in $\left[ \lambda_0, \lambda_0+\lambda \right]$.
\end{proof}

\subsection{Bounding By Rotations}
Define

\begin{equation}
\label{ODErotConj}
 V'_ t = R^{W_t}V_t
\end{equation}
where $R$ is given by 

\begin{equation*}
R = \frac{\lambda} {\sin^2\theta}\begin{bmatrix} -\cos \theta & 1 \\ -1 & \cos \theta \end{bmatrix}, 
\end{equation*}
and $W_t$ is the piecewise constant interpolation of $W_n$, that is $W_t = W_{\lfloor t \rfloor}$. Note that $R$ is chosen so that if we map the UHP to the disk using the version of the Cayley transform sending $z$ to the center of the disk, then $R$ is a rotation about $z$ with speed $\lambda$. For this reason we may think of $R$ as a ``rotation'' even in the UHP. In Theorem \ref{RotThm} we find a relationship between $V_k$ and $V_t$, and in what follows we will use this relationship to understand $V_k$ through $V_t$. This is useful because rotations are relatively simple to deal with. This view of $R$ as a ``rotation'' is also useful in explaining our view of what happens in the projectivization of the $V_t$ process as the point moves past infinity.

\begin{theorem}
\label{RotThm}
 The process $V_k$ is upper-bounded by the process $V_t$ given by differential equation (\ref{ODErotConj}), in the sense that the projectivizations of $V_k$ and $V_t$ are each processes following the point at infinity as it moves along the boundary of the UHP to the right, and for any time $t=k$, the point in the $V_t$ process has moved at least as much as the point in the $V_k$ process has.
\end{theorem}

Consider first a simple version of the $V_k$ process where the $Q$ matrices are unconjugated.  Call this process $\tilde{V}_k$, so

\begin{equation*}
\tilde{V}_k = Q^k \begin{bmatrix} \phi_{1} \\ 0 \end{bmatrix}.
\end{equation*}
Then the $\tilde{V}_k$ process can be described by the finite difference equation

\begin{equation}
\label{FDE}
\tilde{V}_{k + 1} = Q\tilde{V}_{k}
\end{equation}
where we set

\begin{equation*}
\tilde{V}_0 = \begin{bmatrix} v_{1,0} \\ v_{2,0} \end{bmatrix} = \begin{bmatrix} \phi_1 \\ 0 \end{bmatrix}.
\end{equation*}

\begin{lemma}
 Solutions to the finite difference equation (\ref{FDE}) are equal to solutions to differential equation \eqref{ODE} at integer times.
\end{lemma}

\begin{equation}
 \label{ODE}
 \tilde{V}'_t = \begin{bmatrix} 0 & 0 \\ -\lambda & 0 \end{bmatrix} \tilde{V}_t =: \Lambda \tilde{V}_t.
\end{equation}

\begin{proof}
 The difference equation (\ref{FDE}) can be decoupled by considering the rows separately.  The first row gives $\tilde{v}_{1, k+1} = \tilde{v}_{1, k}$.  This means that $\Delta \tilde{v}_1 = 0$ (where we have dropped the $k$ from this coordinate because the solution tells us that it's autonomous).  The second row gives $\tilde{v}_{2, k+1} = -\lambda \tilde{v}_{1, k} + \tilde{v}_{2, k}$.  This means that $\Delta \tilde{v}_2 = -\lambda \tilde{v}_1$, (where again we drop the $k$ because our solution from the first row means that this row is also autonomous). On the other hand, the differential equation (\ref{ODE}) is already decoupled, and encodes precisely the same information: $\tilde{v}_1'=0$, $\tilde{v}_2'=-\lambda \tilde{v}_1$.
\end{proof}
We now consider the differential equation (\ref{ODE}) instead of the difference equation (\ref{FDE}).  We would like to work with the projectivization, specifically the process $\tilde{s} = \tilde{v}_{t,1}/\tilde{v}_{t,2}$.  Using the quotient rule, we obtain the differential equation governing $\tilde{s}$, which is:

\begin{equation}
\label{sODE}
\tilde{s}'=\lambda \tilde{s}^2. 
\end{equation}
Note that $\bar{s}$ gives (through its solutions at integer times) the projectivization of the $\tilde{V}_k$ process. Ultimately we would like to bound the $V_k$ process by the process given in \eqref{ODErotConj}. To that end, we will consider what happens when we replace the matrix $\Lambda$ in \eqref{ODE} by $R$. If we replace $\Lambda$ by $R$ in equation \eqref{ODE}, then with our understanding of $R$ as a rotation, we can use monotonicity to relate the solutions of the two differential equations.

\begin{lemma}
\label{useRot}
 The solution to differential equation (\ref{sODE}) is upper bounded by the solution to the differential equation (\ref{rotODE}), below, which comes from the projectivization of the differential equation obtained by replacing $\Lambda$ with $R$ in the $\tilde{V}_t$ process:
\end{lemma}

\begin{equation}
\label{rotODE}
\tilde{s}'=\frac{\lambda}{\sin^2\theta} \left(\tilde{s}^2 - 2\tilde{s}\cos \theta +1\right) .
\end{equation}

\begin{proof}
 The derivative $\tilde{s}'$ is strictly positive in both differential equations, which means in both cases, the solution $\tilde{s}$ is strictly increasing, so it suffices to show that $\tilde{s}'$ is always bigger in (\ref{rotODE}) than in (\ref{sODE}), or that the ratio
 
 \begin{equation*}
  \frac{\frac{\lambda}{\sin^2\theta} \left(\tilde{s}^2 - 2\tilde{s}\cos\theta +1\right)}{\lambda \tilde{s}^2}
 \end{equation*}
 is always at least 1.  But we can use calculus to find that this ratio is minimized by $\tilde{s}=1/\cos \theta$, and has a minimum value of precisely $1$.
 \end{proof}
 
 At this point we have shown that the simple version of the $V_k$ process ($\tilde{V}_k$, where the $Q$ matrices are unconjugated) has its projectivization upper bounded by the solution to the differential equation given above in \eqref{rotODE}. We will now show that this holds even in the case where the $Q$ matrices are conjugated.
 
 Let $s$ be the projectivization of the process defined by 
 \begin{equation*}
  \tilde{V}_t ' = \Lambda^{W_t} \tilde{V}_t.
 \end{equation*}
  In other words, by using $s$ we are now reintroducing the conjugations.

\begin{corollary}
\label{CorBnd}
 The solution to the differential equation governing $s$ is upper bounded by the solution to the differential equation governing the process corresponding to $s$ but with $\Lambda$ replaced by the rotation matrix $R$. In other words, the result of Lemma \ref{useRot} holds true even in the case where the $Q$ matrices are conjugated.
\end{corollary}

\begin{proof}

 Conjugation of $Q$ by a $k$-independent matrix $W$ is equivalent to replacing the $\tilde{V}$ in the finite difference equation (\ref{FDE}) by $WV$. This new finite difference equation encodes the same information as differential equation (\ref{ODE}) applied to $WV$
 
 \begin{equation*}
  W V'_t = \begin{bmatrix} 0 & 0 \\ -\lambda & 0 \end{bmatrix} W V_t.
 \end{equation*}
 In the projectivization, this means that conjugation of the Q matrices corresponds to applying the transformation $W$ to $\tilde{s}$ in differential equations (\ref{sODE}) and (\ref{rotODE}).  Since $W$ is a fractional linear transformation, it respects order, so the results of Lemma \ref{useRot} still apply. Since $W_t$ is a piecewise constant function, by continuity of the solutions, the bound holds even when conjugating by $W_t$.
 
\end{proof}

We may now prove Theorem \ref{RotThm}:

\begin{proof}
 Equation (\ref{sODE}) with $W_k$ applied to $\tilde{s}$ is the equation governing the projectivization of the process $V_k$, and equation (\ref{rotODE}) with $W_t$ applied to $\tilde{s}$ is the equation governing the projectivization of the process $V_t$. By Corollary \ref{CorBnd} the projectivization of $V_t$ bounds the projectivization of $V_k$.
\end{proof}

Theorem \ref{RotThm} allows us to consider $V_t$ instead of $V_k$ with the effect that the point on the boundary that we are following will always have moved to the right more than it would have without the replacement.  This is useful since $R$, and therefore $R^W$ are rotations, so $R^W$ has a fixed point, $W^{-1}\circ z$.  To figure out where the point $p$ gets moved by the process $V_t$, we need only follow the sequence of centers of rotations: $W_k^{-1} \circ i$.

\subsection{Movement From a Different Perspective}

We will now look at the process $s_t = \mathscr{P}[V_t]$ from the perspective of the process $W_t \circ z$.  From this perspective, $s_t$ will have discrete jumps at integer times.  Write $W_t^{-1} \circ z = X_t + iY_t$ where $X_t$ and $Y_t$ are real and coupled in the following way:  $dY_t = Y_tdZ$ and $dX_t = Y_tdU$ for some processes $U$ and $Z$.  Note that $U$ and $Z$ are pure jump processes.

\begin{lemma}
$V_t$ satisfies the differential equation

\begin{equation*}
V'_ t = \frac{\lambda}{\sin^2 \theta} \begin{bmatrix} - \cos \theta & 1 \\ -1 & \cos \theta \end{bmatrix} ^{\bar{W}_t} V_t  = \hspace{2pt} \frac{\lambda}{\sin \theta} \begin{bmatrix} 0 & 1 \\ -1 & 0 \end{bmatrix} ^{A \bar{W}_t} V_t 
\end{equation*}
\end{lemma}

where 

\begin{equation*}
A = \begin{bmatrix} 1 & -\cos \theta \\ 0 & \sin \theta \end{bmatrix}
\end{equation*}

and

\begin{equation*}
\bar{W}_t = \begin{bmatrix} 1 & -X_t + Y_t \cot \theta \\ 0 & Y_t/\sin \theta \end{bmatrix}.
\end{equation*}

\begin{proof}
The first equality is nearly a restatement of the definition of $V_t$ from equation (\ref{ODErotConj}), but with $\bar{W_t}$ in place of $W_t$, so to prove the first equality it is sufficient to check that $R^{W_t} = R^{\bar{W}_t}$. The eigenvectors of $R$ are 

\begin{equation*}
  \begin{bmatrix} z \\ 1 \end{bmatrix} \text{ and } \begin{bmatrix} \bar{z} \\ 1 \end{bmatrix}.
\end{equation*}
But $W^{-1}_t \circ z = X_t + iY_t$, and we can compute $\bar{W}_t \circ X_t + iY_t = z$, so 

\begin{equation*}
\bar{W}_t^{-1} \begin{bmatrix} z \\ 1 \end{bmatrix} = cW^{-1}_t \begin{bmatrix} z \\ 1 \end{bmatrix}
\end{equation*}
which means that the eigenvectors of $W_t \bar{W}_t^{-1}$ are also 

\begin{equation*}
  \begin{bmatrix} z \\ 1 \end{bmatrix} \text{ and } \begin{bmatrix} \bar{z} \\ 1 \end{bmatrix}
\end{equation*}
so $R$ and $W_t \bar{W}_t^{-1}$ commute. Therefore $R^{W_t \bar{W}_t^{-1}} = R$, and $R^{W_t} = R^{\bar{W}_t}$.
The second equality is true because 

\begin{equation*}
 \begin{bmatrix} -\cos \theta & 1 \\ -1 & \cos \theta \end{bmatrix} = \begin{bmatrix} 0 & 1 \\ -1 & 0 \end{bmatrix}^A. 
\end{equation*}

\end{proof}
Now let $F_t$ be $V_t$ seen from the perspective of the $X_t + iY_t$, so we have

\begin{equation*}
 F_t := A \bar{W}_t V_t = \begin{bmatrix} v_{1,t}-X_t v_{2,t} \\ Y_t v_{2,t} \end{bmatrix}
\end{equation*}
and we can compute $dF_t$ as follows:

\begin{align*}
dF_t &= Y_t\begin{bmatrix} -dU \\ dZ \end{bmatrix} v_{2,t} + \begin{bmatrix} v'_{1,t} - X_t v'_{2,t} \\ Y_t v'_{2,t} \end{bmatrix} \\
&= Y_t\begin{bmatrix} -dU \\ dZ \end{bmatrix} v_{2,t} + \begin{bmatrix} 1 & -X_t \\ 0 & Y_t \end{bmatrix} V'_t dt\\
&= Y_t\begin{bmatrix} -dU \\ dZ \end{bmatrix} v_{2,t} + A\bar{W}V'_t dt\\
&= Y_t\begin{bmatrix} -dU \\ dZ \end{bmatrix} v_{2,t} + \frac{\lambda}{\sin \theta} \begin{bmatrix} 0 & 1 \\ -1 & 0 \end{bmatrix} F_t dt.
\end{align*}
Once again the differential equation is autonomous, so can be written compactly as:

\begin{equation}
\label{FMDE}
dF = F_2 \begin{bmatrix} -dU \\ dZ \end{bmatrix} + \frac{\lambda}{\sin \theta} \begin{bmatrix} 0 & 1 \\ -1 & 0 \end{bmatrix} F dt
\end{equation}
and taking projectivizations, we define 

\begin{equation*}
\bar{s}_t :=  \frac{F_1}{F_2}.
\end{equation*}

\begin{remark}
 The process $\bar{s}$ starts at $p$ and moves along the boundary of the UHP, however it is not well defined because of the discrete jumps at integer times. To ensure that $\bar{s}$ is well defined, we will always use the right-continuous version of the process.
\end{remark}

\begin{lemma}
\label{bktkblup}
Fix $\lambda$, $M$, $\epsilon$, and $\beta \leq \left(2M\right)^{-1}$. Let $X_t$ and $Y_t$ be real processes coupled by $dY_t = Y_tdZ$ and $dX_t = Y_tdU$, where $U$ and $Z$ are pure jump processes. If $|\Delta X_t|/Y_t \leq M$ (for all $t$), and the process $\log Y_n + \left[\left(\epsilon + \lambda\beta\right)/\sin\theta + 2M\beta\right]n$ has no backtracks as large as $\log \left( \epsilon\beta/\lambda \right)$, then the process $\bar{s}$ can never pass $\infty$.
\end{lemma}

\begin{proof}
First define 

\begin{equation*}
L :=  \log\left(-\bar{s}\right) = \log \left(-F_1\right) - \log F_2
\end{equation*}
This doesn't make sense for $\bar{s} \geq 0$, but for the remainder of the proof we will only be concerned with negative values of $\bar{s}$, so this causes no problems. We can use \eqref{FMDE} to find the differential equation governing $L$. This differential equation will have three terms, the first two of which come from jumps:

\begin{itemize}
 \item $dF_1/dU = -F_2$ and $dF_2/dU = 0$. When $F_1 \rightarrow F_1 - F_2 dU$, $\log(-F_1) \rightarrow \log (-(F_1 - F_2 dU))$, so $dL = \log(-(F_1 - F_2 dU)) - \log(-F_1) = \log(1-dU/\bar{s})$. So $dL$ has a $\log(1-dU/\bar{s})$ term.
 
 \item $dF_2/dZ = F_2$ and $dF_1/dZ = 0$. When $F_2 \rightarrow F_2 + F_2 dZ$, $\log(F_2) \rightarrow \log (F_2 + F_2 dZ)$, so $dL$ has a $-\log(1+dZ)$ term.
 
 \item At non-integer values of $t$, $L$ is continuous in $t$, so we may use the quotient rule to compute that $dL$ has a $\frac{\lambda}{\sin \theta} (\bar{s} + 1/\bar{s})dt$ term.
\end{itemize}
So the differential equation governing $L$ is

\begin{equation*}
dL = \frac{\lambda}{\sin \theta}\left(\bar{s}+1/\bar{s}\right)dt - \log\left(1+dZ\right) + \log \left(1-dU/\bar{s}\right)
\end{equation*}
and if we integrate both sides between $t_1^-$ and $t_2$ we get

\begin{equation}
\label{solveMixODE}
L_{t_1^-}-L_{{t_2}} = \int_{t_1}^{t_2} \frac{\lambda}{\sin \theta} \left(e^L + e^{-L}\right)dt + \int_{t_1^-}^{t_2} \log\left(1+dZ\right) - \int_{t_1^-}^{t_2} \log\left(1- \frac{dU}{\bar{s}}\right).
\end{equation}
Here, the second and third integral correspond to summing the integrands over the jumps of $Z$ and $U$. Also, note that both sides absorbed a negative sign. Now let $t_2 = \inf \{t: \bar{s} \geq -1/\beta \} $, and let $t_1 = \sup_{t<t_2} \{t: \bar{s} \leq - \epsilon / \lambda \}$.  Then we have the following inequalities: 

\begin{equation*}
\label{LLowB}
 L_{t_1^-} \geq \log \epsilon / \lambda
\end{equation*}

\begin{equation*}
\label{LUpB}
L_{t_2} \leq \log 1/\beta  
\end {equation*}
so that

\begin{equation}
\label{LbndB}
L_{t_1^-} - L_{t_2} \geq \log \epsilon/\lambda - \log 1/\beta.  
\end {equation}
When $t_1 \leq t \leq t_2$ we have:

\begin{equation}
\label{expLbp}
\epsilon/\lambda \geq e^L \geq 1/\beta
\end {equation}
and

\begin{equation}
\label{expLbn}
\lambda/\epsilon \leq e^{-L} \leq \beta.
\end {equation}
Since $Y_t$ is piecewise constant $dZ = 0$ at non-integer times, so $Y_{t+1} - Y_t = Y_t dZ$ by the definition of $Z$, meaning $dZ+1 = Y_{t+1}/Y_t$ at integer times. Hence 

\begin{equation}
\label{Zint}
\int_{t_1^-}^{t_2} \log\left(1+dZ\right) = \log \left(Y_{t_2}/Y_{t_1^-}\right) = \log Y_{t_2} - \log Y_{t_1^-} .
\end{equation}
Since $\Delta U$ is upper bounded by $M$, $-\bar{s}$ is lower bounded by $1/\beta$ on the interval we are considering, and $\beta \leq \left(2M\right)^{-1}$, we have $|dU/\bar{s}| \leq M\beta \leq 1/2$. For $x \leq 1/2$ we can use $-\log \left(1-x\right) < 2x$ to get

\begin{equation}
\label{Uint}
-\int_{t_1^-}^{t_2} \log\left(1 - dU/\bar{s}\right) \leq \left(\lfloor t_2 \rfloor - \lfloor t_1^- \rfloor\right) 2M \beta .
\end{equation}
We are now able to continue integrating in equation (\ref{solveMixODE}). Combining (\ref{LbndB}) -- (\ref{Uint}), (\ref{solveMixODE}) implies that

\begin{equation*}
\log \epsilon \beta / \lambda \leq \left(t_2 - t_1\right) \frac{\lambda}{\sin \theta} \left(\epsilon/\lambda + \beta\right) + \log Y_{t_2} - \log Y_{t_1^-} + \left(\lfloor t_2 \rfloor - \lfloor t_1^- \rfloor\right)2M \beta
\end{equation*}
and by rearranging, we have:

\begin{equation*}
\log Y_{t_2} - \log Y_{t_1^-} +\left(t_2 - t_1\right)\left[\left(\epsilon + \lambda \beta\right)/\sin \theta\right] + \left(\lfloor t_2 \rfloor - \lfloor t_1^- \rfloor\right)2M \beta \geq \log \epsilon \beta / \lambda.
\end{equation*}
For this inequality to hold, the process $\log Y_n + \left[\left(\epsilon + \lambda \beta\right)/\sin \theta + 2M \beta\right]n$ must have a backtrack of size at least $\log \epsilon \beta / \lambda$ between $\lfloor t_1^- \rfloor$ and $\lceil t_2 \rceil$.  So such backtracks are necessary in order for $\bar{s}$ to move through through the range between $-\epsilon/\lambda$ to $-1/\beta$, which is necessary for $\bar{s}$ to pass $\infty$. In particular, we get the condition that in order for $\bar{s}$ to pass $\infty$, the process $\log Y_n + \left[\left(\epsilon + \lambda \beta\right)/\sin \theta + 2M\beta\right]n$ must backtrack by at least $\log \epsilon\beta/\lambda$.

\end{proof}

\subsection{Proof of Theorem \ref{tmbktkblp}}

\begin{proof}

Define $N_n$ to be the number of eigenvalues of $H_{\omega,n}$ in the interval $\left[\lambda_0, \lambda_0 + \lambda\right]$.  By Lemma \ref{keyObv}, $N_n$ is equal to the number of times the process $\{\mathscr{P}\left[ V_k \right] \}_{k=1}^n$ passes the point $\mathscr{P}[v_*]$, and so from Theroem \ref{RotThm} we get that $N_n$ is less than or equal to the number of times the process $s_t$ passes the point $\mathscr{P}[v_*]$, which is no more than $1$ plus the number of times the process $s_t$ passes $\infty$.

Lemma \ref{bktkblup} tells us that in order for the process $\bar{s}$, and therefore the process $s_t$ to pass $\infty$, there must be a backtrack as large as $\log \epsilon\beta/\lambda$ in the process $\log Y_n + \left[\left(\epsilon + \lambda \beta\right)/\sin \theta + 2M\beta\right]n$. 
\end{proof}


Theorem \ref{tmbktkblp} gives a deterministic result relating the number of eigenvalues of a finite level schrodinger operator to the number of large backtracks of the imaginary part of a random walk. It also relies on the existence of a bound on the jumps of the real part of that random walk. We now prove that such a bound exists.

\subsection{Bounding The Real Part}

\begin{theorem}
\label{MBound}
 Let $X_n$ and $Y_n$ be defined as in Section \ref{MoTM}, with $\sigma \in [0,1]$, $\theta$ arbitrary, $|\omega_i| \leq c_0$ and $c_0 \geq 1$. Then for all $k \geq 0$
  
 \begin{equation*}
  \frac{|X_{k+1}-X_k|}{Y_k} \leq \frac{\sqrt{5}}{2} \frac{\sigma c_0^2}{\sin^2 \theta}.
 \end{equation*}
\end{theorem}

\begin{proof}
 Define
 
 \begin{equation}
 \label{d1}
  d_1(x+iy, x' + iy') = \frac{|x-x'|}{y}
 \end{equation}
 and also
 
 \begin{equation*}
 \label{d2}
  d_2(x+iy, x' + iy') = \frac{(x-x')^2 + (y-y')^2}{yy'}.
 \end{equation*}
 
 \begin{lemma}
  $d_2$ is invariant under M\"obius transforms, namely
  
  \begin{equation}
   d_2(z,z') = d_2(Tz,Tz')
  \end{equation}
 for any $T$ fixing the UHP.
  
 \end{lemma}

 \begin{proof}
 
It suffices to check the following 3 cases:

$d_2$ is invariant under shifts:
\begin{align*}
 d_2( z+d, z'+d) = \frac{\left( (x+d) - (x'+d) \right)^2 + (y-y')^2}{yy'} = d_2( z, z')
\end{align*}

$d_2$ is invariant under dialations:
\begin{align*}
 d_2( \alpha z, \alpha z') = \frac{\alpha^2 (x-x')^2 + \alpha^2(y-y')^2  }{\alpha y \alpha y'} = d_2( z, z')
\end{align*}

$d_2$ is invariant under inversion:
\begin{align*}
 d_2( 1/z, 1/z') &= d_2(\frac{x-iy}{|z|^2},\frac{x'-iy'}{|z'|^2}) \\
 &= \frac{ \left( x/|z|^2 - x'/|z'|^2 \right)^2 + \left( -y/|z|^2 + y'/|z'|^2 \right)^2}{yy'/|z|^2 |z|^2} \\
 &= \frac{|z|^2|z'|^2}{yy'} \left[ \frac{x^2}{|z|^4} - \frac{2xx'}{|z|^2|z'|^2} + \frac{(x')^2}{|z'|^4} + \frac{y^2}{|z|^4} - \frac{2yy'}{|z|^2|z'|^2} + \frac{(y')^2}{|z'|^4}   \right] \\
 &= \frac{1}{yy'} \left[ (x^2+y^2) \frac{|z'|^2}{|z|^2} + ((x')^2 + (y')^2)\frac{|z|^2}{|z'|^2} - 2(xx'+yy') \right] \\
 &= \frac{1}{yy'} \left[ |z'|^2 + |z|^2 - 2(xx'+yy') \right]\\
 &= \frac{(x-x')^2 + (y-y')^2}{yy'} = d_2( z, z').
\end{align*}

 \end{proof}
 
 \begin{lemma}
  \begin{equation*}
   d_1^2 \leq d_2(1+\frac{d_2}{4})
  \end{equation*}
 \end{lemma}

 \begin{proof}
 Write $z = x+iy$, $z' = x'+iy'$. Since both $d_1$ and $d_2$ are invariant under shifts and dialations of the UHP, we may assume that $x = 0$ and $y = 1$. Then
 
 \begin{equation*}
  d_1(z,z') = |x'|
 \end{equation*}
and

  \begin{equation*}
  d_2(z,z') = \frac{(x')^2 + (1-y')^2}{y'}.
 \end{equation*}
Now we can simplify:

\begin{equation*}
 d_2(z,z') \left( 1 + \frac{d_2(z,z')}{4}  \right) - (x')^2 = \frac{((x')^2+1-(y')^2)^2}{(4y')^2} \geq 0
\end{equation*}
so that

\begin{equation*}
 d_2(1+\frac{d_2}{4}) \geq (x')^2 = d_1^2
\end{equation*}
completing the proof.
 \end{proof}
Now we have the following:
 
 \begin{equation}
 \label{Mformula}
  \frac{|X_k - X_{K+1}|}{Y_k} = d_1 \left(W_k^{-1} \circ z, W_{k+1}^{-1} \circ z \right) \leq \sqrt{d_2 \left(W_k^{-1} \circ z, W_{k+1}^{-1} \circ z \right) \left( 1 + d_2 \left(W_k^{-1} \circ z, W_{k+1}^{-1} \circ z \right)/4 \right) }.
 \end{equation}
But we can bound $d_2 \left(W_k^{-1} \circ z, W_{k+1}^{-1} \circ z \right)$ as follows:

\begin{align*}
 d_2 \left(W_k^{-1} \circ z, W_{k+1}^{-1} \circ z \right) &= d_2 \left(W_k^{-1} \circ z, W_{k}^{-1} T_{k+1}^{-1} \circ z \right) \\
 &= d_2(z, T_{k+1}^{-1} \circ z) \\
 &= d_2(T_{k+1} \circ  z, z).
\end{align*}
When $\omega = 0$ we have that $T_{k+1}^{\omega=0} \circ z = z$, so

\begin{align*}
 d_2(T_{k+1} \circ  z, z) &= d_2(T_{k+1} \circ  z, T_{k+1}^{\omega=0} \circ z) \\
 &= d_2 \left( \frac{(\lambda_0-\sigma \omega_{k+1})z-1}{z}, \frac{\lambda_0 z - 1}{z}  \right)\\
 &= d_2 \left( \lambda_0 - \sigma \omega_{k+1} - \bar{z}, \lambda_0 - \bar{z}  \right).
\end{align*}
 By invariance under M\"obius transforms, this is equal to
 
 \begin{equation*}
  d_2 \left( - \sigma \omega_{k+1} + i\sin \theta, i \sin \theta  \right)
 \end{equation*}
which can be computed to get

\begin{equation*}
 d_2 \left(W_k^{-1} \circ z, W_{k+1}^{-1} \circ z \right) = \frac{(\sigma \omega_{k+1})^2}{\sin^2\theta}.
\end{equation*}
Using this bound in \eqref{Mformula} gives

\begin{equation*}
 \frac{|X_{k+1}-X_k|}{Y_k} \leq \sqrt{\frac{(\sigma c_0)^2}{\sin^2\theta} \left( 1 + \frac{(\sigma c_0)^2}{4 \sin^2 \theta} \right)  }
\end{equation*}
and since we have $\sin\theta \leq 1$, $c_0 \geq 1$, and $\sigma \leq 1$, we get
\begin{equation*}
 \frac{|X_{k+1}-X_k|}{Y_k} \leq \frac{\sqrt{5}\sigma c_0^2}{\sin^2\theta}.
\end{equation*}
 
\end{proof}

\section{Bounding Backtracks}
\label{sect4}

%
%
%
%
%
%
%

\subsection{The Figotin-Pastur Vector}

\begin{lemma}
\label{cuteTrick}
 Let $\tilde{M}$ be a $2 \times 2$ matrix with determinant $1$.  Then

 \begin{equation*}
  \operatorname{Im} \left(\tilde{M}^{-1} \circ i \right) = \left \| \tilde{M} \begin{bmatrix} 1 \\ 0 \end{bmatrix} \right\|^{-2}.
 \end{equation*}

\end{lemma}

\begin{proof}

Write

 \begin{equation*}
  \tilde{M} =  \begin{bmatrix} a & b \\ c & d \end{bmatrix}
 \end{equation*}
so that we have

 \begin{align*}
  \operatorname{Im} \left(\tilde{M}^{-1} \circ i\right) &= \operatorname{Im} \left( \begin{bmatrix} d & -b \\ -c & a \end{bmatrix} \circ i\right) \\
  &= \operatorname{Im} \frac{id-b}{-ic+a}\\
  &= \frac{\operatorname{Im}\left(\left(id-b\right)\left(a+ic\right)\right)}{a^2+c^2} \\
  &= \frac{1}{a^2+c^2}\\
  &= \left \| \tilde{M} \begin{bmatrix} 1 \\ 0 \end{bmatrix} \right\|^{-2}.
 \end{align*}

\end{proof}

We want to understand the backtracks of the $\log Y_t$ process, which means we we want to follow the log of $\operatorname{Im} \left(\left(A \bar{W}_t\right)^{-1} \circ i\right)$.
Lemma (\ref{cuteTrick}) allows us to instead follow $1/ ||\gamma_t||^2$, where

\begin{equation*}
\gamma_t :=  A\bar{W}_t \begin{bmatrix} 1 \\ 0 \end{bmatrix}
\end{equation*}
which is the well-known Figotin-Pastur vector for which a recurrence relation is known.
Define

\begin{equation*}
  \mathscr{P}[\gamma_k] = \sqrt{r_k} e^{i \alpha_k}
\end{equation*}
so that

\begin{equation*}
  \mathscr{P}[Y_k^{-1}] = r_k = ||\gamma_k||^2
\end{equation*}
and recall that

\begin{equation*}
  z = e^{i\theta}
\end{equation*}
and

\begin{equation*}
  \rho = \frac{1}{2\sin \theta} = \frac{1}{|1-z^2|}.
\end{equation*}
Then from \cite{BS2} we have the recurrence relations

\begin{equation}
\label{rkrecur}
 r_{k+1} = r_{k} ( 1 + 2 \sigma^2 \omega^2_{k+1} \rho^2 +  2 \sigma \omega_{k+1} \rho \sin \left( 2\alpha_{k}+2\theta \right) - 2 \sigma^2 \omega^2_{k+1} \rho^2 \cos \left(2\alpha_{k}+2\theta \right))
\end{equation}
and

\begin{equation}
\label{angRecur}
 e^{2 i \alpha_{k+1}} = e^{2 i\alpha_{k}} z^2 + \frac{\sigma \omega_{k+1} i \rho \left(z^2 e^{2i\alpha_{k}} - 1\right)^2 }{1 + \sigma \omega_{k+1} i \rho \left( 1 - z^2 e^{2i\alpha_{k}}\right)}
\end{equation}
and the non-recursive expression for $r_k$

\begin{equation*}
 r_{k} = \prod_{j=1}^{k-1} \left(1 + 2 \sigma^2 \omega^2_{j} \rho^2 +  2 \sigma \omega_{j} \rho \sin\left(2\alpha_{j-1}+2\theta\right) - 2 \sigma^2 \omega^2_{j} \rho^2 \cos \left(2\alpha_{j-1}+2\theta\right)\right).
\end{equation*}

\subsection{Martingales}
In what follows, we will use a martingale argument to bound the probability of a large backtrack of the process $\log Y_n + \kappa n$, with $Y_n$ as in the previous section and $\kappa$ sufficiently small. We will use a function of $Y_n$ which, raised to the power of $1-\delta$, is a supermartingale for an appropriate choice of $\delta$. This $\delta$ will need to be big enough to make the process a supermartingale, but it can't be too large or else it will ruin the bound we are trying to get. We find lower and upper bounds for $\delta$; the lower bound is the more important bound, necessary to ensure we are working with a supermartingale, where as the upper bound we choose is for technical reasons, specifically to bound a Taylor expansion cutoff, and could be chosen differently if desired.

\begin{lemma}
 \label{smgale}
 Assume there are positive constants $c_1\ldots c_7$ so that the following holds. 
 Let $X_k$ be a sequence of random variables such that
 \begin{equation*}
 E\left(X_k|\mathcal{F}_{k-1}\right) = \sigma^2 B_{k-1},\qquad
  E\left(X_k^2|\mathcal{F}_{k-1}\right) = \sigma^2 A_{k-1},
 \end{equation*}
 where $|A_k|\le 9c_0 \rho^3$, $|B_k|\le 4\rho^2$, and $|X_k|\le c_1\sigma$, and where $\mathcal{F}_{k}$ is the sigma algebra generated by $\omega_1, \dots, \omega_k$.  Assume further that there exists a constant $\tilde c$ and some functions $F_k, G_k$ such that with $\Delta F_k=F_k-F_{k-1}$ we have
 \begin{equation}
 \label{Fdef}
|B_k-\Delta F_k - \tilde{c} | \le c_3 \sigma,\qquad |A_k - \Delta  G_k - \tilde{c} | \le c_5\sigma,
 \end{equation}
 and
\begin{equation}\label{deltaFGBounds}
|\Delta F_k| \le c_2, \qquad |\Delta G_k|\le c_4.
\end{equation}
Then for $\kappa \in[0,1]$, $\sigma$ satisfying
\begin{equation}
\label{pickSig}
\sigma \le \max(c_1,c_6,(c_2+c_4)^{1/2})^{-1}
\end{equation}
and for $\delta$ satisfying

 \begin{equation}\label{asdelta}
\frac{2\sigma}{\tilde{c}}\left(\frac{2\kappa}{\sigma^3} + c_3 + c_5 + 2c_7\right)\leq \delta \leq \frac12
 \end{equation}
with $c_7$ as in \eqref{c6}, the following process is a supermartingale:

 \begin{equation*}
  \Pi_k = e^{\sigma^2 \left(1 - \delta \right)\left(F_{k-1} - \left(1-\delta/2\right)G_{k-1}\right)} \prod_{i=1}^k \left(e^{-\kappa} \left( 1 + X_i \right)\right)^{\delta-1}.
 \end{equation*}
\end{lemma}

\begin{proof}

\begin{align*}
E(\Pi_k |\mathcal{F}_{k-1}) =
 \Pi_{k-1} &E(e^{\sigma^2(1-\delta) (\Delta F_{k-1} - (1-\delta/2)\Delta G_{k-1})}(e^{-\kappa}(1+X_k))^{\delta-1}|\mathcal{F}_{k-1})
\end{align*}
We will write
\begin{align}\notag
1+a:=E((1+X_k)^{\delta-1}|\mathcal{F}_{k-1}), \qquad 1+b:= e^{\sigma^2(1-\delta) (\Delta F_{k-1} - (1-\delta/2)\Delta G_{k-1})}
\end{align}
and it suffices to show that

\begin{equation*}
(1+a)(1+b)\le e^{-\kappa}.
\end{equation*}
First we get two bounds on $a$:

For $\delta\in[0,1/2]$ and $|x|\le1/ 4$, Taylor expansion gives $|(1+x)^{\delta-1} -1|\le2|x|$, giving the bound

\begin{equation}\label{b1}
|a|\le 2c_1\sigma.
\end{equation}
Taking the Taylor expansion one term further gives

\begin{equation*}
 (1+x)^{\delta-1} \le 1-(1-\delta)(x -(1-\delta/2) x^2) + 3|x|^3.
\end{equation*}
Since $|X_k|\le c_1\sigma \le 1/4$ we get the more precise bound on $a$:

\begin{equation}\label{b2}
a \le -\sigma^2(1-\delta)(B_{k-1} - (1-\delta/2)A_k)+3c_1^3\sigma^3.
\end{equation}
Now we get a bound on $b$:

For $|x|\le 1$ we have the two inequalities $|e^x-1|\le 2|x|$ and $e^x\le 1+x+x^2$. Note that by \eqref{deltaFGBounds} we have $|\Delta F|+ |\Delta G|\le c_2+c_4$. The first inequality gives that for $\sigma^2\le 1/(c_2+c_4)$ we have the bound on $b$:

\begin{equation}\label{b3}
 |b|\le2 (c_2+c_4)\sigma^2.
\end{equation}
The second inequality gives more precisely:

\begin{equation}\label{b4}
b\le \sigma^2(1-\delta) (\Delta F_{k-1} - (1-\delta/2)\Delta G_{k-1}) +
\sigma^4 (c_2+c_4)^2.
\end{equation}
If $\sigma<1/c_6$, the last term is at most $\sigma^3(c_2+c_4)^2/c_6$. To bound the product $(1+a)(1+b)$ we use the finer bounds for $a+b$ and the rough bounds for $|ab|$. Combining (\ref{b1},\ref{b2},\ref{b3},\ref{b4}) this way, we get an upper bound of

\begin{equation}\label{bound3}
1+\sigma^2(1-\delta)(\Delta F_{k-1} - B_{k-1} + (1-\delta/2)(A_{k-1}-\Delta G_{k-1}))+\mbox{error}
\end{equation}
where

\begin{equation}\label{c6}
\mbox{error} \le (3c_1^3 + (c_2+c_4)^2/c_6+ 4c_1(c_2+c_4))\sigma^3 := c_7\sigma^3.
\end{equation}
Now by assumption \eqref{Fdef}, the quantity \eqref{bound3} is at most

\begin{equation*}
1+\sigma^2(1-\delta)\left(c_3 \sigma +c_5\sigma -\delta\tilde c/2 \right)+c_7\sigma^3
\end{equation*}
where the term in the brackets is negative by the lower bound in \eqref{asdelta}, so by the upper bound in  \eqref{asdelta} we get that

\begin{equation*}
1+\frac{\sigma^2}{2} \left(c_3 \sigma +c_5\sigma -\delta\tilde c/2 \right)+c_7\sigma^3  \le 1-\kappa\le e^{-\kappa},
\end{equation*}
where the first inequality is equivalent to the left inequality of \eqref{asdelta}. This completes the proof.
\end{proof}

We will assume (and heavily use) for the rest of the paper that
\begin{equation}\label{sigma}
  \sigma\le \frac{2\sin\theta |\sin 2\theta|}{10c_0^3}, \mbox{ implying }\quad \sigma \le \frac{4\sin^2 \theta}{10c_0^3}= \frac1{10\rho^2c_0^3}, \qquad \sigma \le  \frac{\sin \theta}{5c_0^3}=\frac1{10\rho c_0^3} \leq \frac1{10\rho c_0}.
\end{equation}
The last inequality, combined with the fact that $c_0$, an absolute bound on a random variable of variance $1$, satisfies

\begin{equation*}
 c_0 \leq 1
\end{equation*}
gives

\begin{equation}\label{sigmaBetter}
 \sigma c_0 \rho \leq \frac1{10}.
\end{equation}

\begin{lemma}
\label{supMgale}
If $E(\omega_j) = 0$, $E(\omega_j^2) = 1$, $|\omega_j| \le c_0$, then there exist functions $F_k$ and $G_k$ satisfying
$$
|F_k|\le 4\rho^3, \quad |G_k|\le \frac{2\rho^2}{|\sin 2\theta|}
$$
so that for $\sigma$ satisfying \eqref{sigma}, $\kappa\in[0,1]$ and $\delta$ satisfying
\begin{equation*}
\frac{\kappa}{\sigma^2 \rho^2} + 224 \frac{c_0^3\rho\sigma}{|\sin 2\theta|}  \leq \delta \leq \frac12
\end{equation*}
we have that with

\begin{equation*}
r_k= \prod_{j=1}^{k-1} \left(1 + 2 \sigma^2 \omega^2_{j} \rho^2 +  2 \sigma \omega_{j} \rho \sin(2\alpha_{j-1}+2\theta)\\
 - 2 \sigma^2 \omega^2_{j} \rho^2 \cos (2\alpha_{j-1}+2\theta)\right)
\end{equation*}
the following process is a supermartingale

\begin{equation}
 e^{(F_{k-1} - (1-\delta/2)G_{k-1})\sigma^2 (1 - \delta)} (e^{-\kappa k} r_k)^{(\delta-1)}.
\end{equation}
\end{lemma}

\begin{proof}
First compute
\begin{multline}\label{erv}
 E \left(2 \sigma^2 \omega^2_{j} \rho^2 +  2 \sigma \omega_{j} \rho \sin(2\alpha_{j-1}+2\theta) - 2 \sigma^2 \omega^2_{j} \rho^2 \cos (2\alpha_{j-1}+2\theta) \right) =\\ 2\sigma^2 \rho^2 (1-\cos(2\alpha_{j-1} + 2\theta))
\end{multline}
and define
\begin{equation}
\label{Bdef}
B_{i-1} = 2\rho^2 (1-\cos(2\alpha_{i-1} + 2\theta)).
\end{equation}
Clearly
\begin{equation*}
|B_{i-1}|\le 4\rho^2.
\end{equation*}
Moreover, the random variable in \eqref{erv} is absolutely bounded above by
\begin{equation*}
  4\sigma^2c_0^2\rho^2+2\sigma c_0\rho \le \frac{12}{5} \sigma c_0\rho=:c_1\sigma
\end{equation*}
where the inequality comes from \eqref{sigmaBetter}. Write $\Sigma = \sum_{j=1}^{k} e^{2i\alpha_{j}} $, and sum (\ref{angRecur}) between $1$ and $k-1$ to get

\begin{equation*}
 \Sigma - e^{2i\alpha_1} = z^2 (\Sigma - e^{2i\alpha_k}) + \sigma \sum_{j=1}^{k-1}  \frac{ \omega_{j+1} i \rho (z^2 e^{2i\alpha_{j}} - 1)^2 }{1 + \sigma \omega_{j+1} i \rho ( 1 - z^2 e^{2i\alpha_{j}})}.
\end{equation*}
Call the sum on the right $\tilde{\Sigma}$. By \eqref{sigmaBetter}, $\sigma\rho|\omega_j|\le 1/10$, and the denominator is bounded below in absolute value by $4/5$.
The terms in $\tilde {\Sigma}$ are thus bounded above in absolute value by
$\frac{4c_0\rho }{4/5}=5c_0\rho$.
Rearranging gives

\begin{equation*}
 \Sigma = \frac{e^{2i\alpha_1} - z^2 e^{2i\alpha_k} + \sigma \tilde{\Sigma} }{1-z^2}
\end{equation*}
and multiplying everything by $-2\rho^2z^2 = -2\rho^2 e^{2i\theta}$ and taking the real part of both sides gives

\begin{equation*}
-2\rho^2 \sum_{j=1}^{k} \cos (2\alpha_j+2\theta) = -2\rho^2\operatorname{Re} z^2 \frac{e^{2i\alpha_1} - z^2 e^{2i\alpha_k}}{1-z^2} -2\rho^2\operatorname{Re} z^2 \frac{\sigma \tilde{\Sigma}}{1-z^2}.
\end{equation*}
Call the first term on the right hand side $F_k$. We have
$$
|\Delta F_k|\le \frac{4\rho^2}{|1-z^2|}=4{\rho^3}=:c_2, \qquad |F_k|\le \frac{4\rho^2}{|1-z^2|}=4{\rho^3}.
$$
Moreover we have
\begin{equation*}\label{xx}
|B_k-\Delta F_k - 2 \rho^2 | = |2\rho^2\operatorname{Re} z^2 \frac{\sigma \Delta\tilde{\Sigma_k}}{1-z^2} |
\le 10 c_0\rho^4\sigma  =:c_3\sigma .
\end{equation*}
Now compute
\begin{multline}
 E((2 \sigma^2 \omega^2_{j} \rho^2 +  2 \sigma \omega_{j} \rho \sin(2\alpha_{j-1}+2\theta) - 2 \sigma^2 \omega^2_{j} \rho^2 \cos (2\alpha_{j-1}+2\theta))^2)\\
 \leq 16 c_0^3 \rho^3 \sigma^3 (1+c_0 \rho) + 4 \rho^2 \sigma^2 \sin^2 (2\alpha_{j-1}+2\theta) \\
 = 16 c_0^3 \rho^3 \sigma^3 (1+c_0 \rho) + 2 \rho^2 \sigma^2 - 2 \rho^2 \sigma^2 \cos (4\alpha_{j-1}+4\theta)
\end{multline}
and define

\begin{equation}
\label{Adef}
A_{i-1} = 16 c_0^3 \rho^3 \sigma (1+c_0 \rho) + 2 \rho^2 - 2 \rho^2 \cos (4\alpha_{j-1}+4\theta)
\end{equation}
which is upper bounded as
$$
A_{i-1}\le
16 c_0^3 \rho^3 \sigma (1+c_0 \rho ) + 4 \rho^2 \le 16 c_0^3 \rho^3 \sigma 3c_0 \rho + 4c_0\rho^3 \le (\frac{16c_0^2 \cdot 3}{10} + 4) c_0 \rho^3\le  9c_0\rho^3
$$
again using \eqref{sigmaBetter}.   
Write $\Sigma = \sum_{j=1}^{k} e^{4i\alpha_{j}} $, and square both sides of (\ref{angRecur}), then sum from $1$ to $k-1$ to get

\begin{multline*}
 \Sigma - e^{4i\alpha_1} = z^4 (\Sigma - e^{4i\alpha_k}) + \sigma  \sum_{j=1}^{k-1} \Big( \sigma\frac{ - \omega_{j+1}^2 \rho^2 (z^2 e^{2i\alpha_{j}} - 1)^4 }{(1 + \sigma \omega_{j+1} i \rho ( 1 - z^2 e^{2i\alpha_{j}}))^2}
 + 2 e^{2i\alpha_k}z^2  \frac{\omega_{k+1}i\rho(z^2 e^{2i\alpha_k}-1)}{1-\sigma\omega_{k+1}i\rho(1-z^2e^{2i\alpha_k})}  \Big).
\end{multline*}
Call the sum on the right $\tilde{\Sigma}$. The terms in $\tilde \Sigma$ are bounded by

\begin{equation*}
\sigma c_0^2\rho^22^4/(4/5)^2 + 4c_0\rho/(4/5)\le 8 c_0 \rho
\end{equation*}
again making use of \eqref{sigmaBetter} multiple times. Rearranging gives
\begin{equation}
 \Sigma = \frac{e^{4i\alpha_1} - z^4 e^{4i\alpha_k} + \sigma \tilde{\Sigma} }{1-z^4}
\end{equation}
and multiplying everything by $-2\rho^2 z^4$ and taking the real part of both sides gives

\begin{equation}
-2\rho^2  \sum_{j=1}^{k} \cos (4\alpha_j+4\theta) = -2\rho^2 \operatorname{Re} z^4 \frac{e^{4i\alpha_1} - z^4 e^{4i\alpha_k}}{1-z^4} -2\rho^2 \operatorname{Re} z^4 \frac{\sigma \tilde{\Sigma}}{1-z^4}.
\end{equation}
Call the first term on the right hand side $G_k$.
We have
$$
|\Delta G_k|\le \frac{4\rho^2}{|1-z^4|}=\frac{2{\rho^2}}{|\sin 2\theta|}=:c_4, \qquad |G_k|\le \frac{2{\rho^2}}{|\sin 2\theta|}$$
Moreover we have

\begin{equation*}\label{xx2}
|A_k-\Delta G_k - 2 \rho^2 | = |2\rho^2\operatorname{Re} z^4 \frac{\sigma \Delta\tilde{\Sigma_k}}{1-z^4} |
\le \frac{8 c_0\rho^3}{|\sin 2\theta|}\sigma  =:c_5\sigma.
\end{equation*}
We now collect constants:
$$
c_1=\frac{12}{5}c_0 \rho , \quad c_2=4\rho^3, \quad c_3 = 10 c_0\rho^4, \quad c_4=\frac{2 \rho^2}{|\sin 2\theta|}, \quad
c_5= \frac{8c_0 \rho^3}{|\sin 2 \theta|}, \qquad \tilde{c} = 2\rho^2, \quad c_6 := \frac{10 c_0^3}{2\sin\theta |\sin 2 \theta|}. 
$$
We now apply Lemma \ref{smgale}. The condition \eqref{pickSig} on $\sigma$ is easily satisfied by \eqref{sigma}.
For the condition \eqref{asdelta}, we use the inequality $1/2 \le \rho \le 1/|\sin 2\theta|$ and $1\le c_0$ to get the bound
\begin{multline*}
c_3 + c_5 + 6c_1^3 + 2(c_2+c_4)^2/c_6+ 8c_1(c_2+c_4))\le \\
\frac{c_0^3\rho^3}{|\sin 2\theta|}
\left(10+8+ 6(\frac{12}{5})^3+\frac2{10}(4+2)^2+2\frac{48}{5}(4+2)
\right).
\end{multline*}


The constant above is less than $224$. The claim follows.
\end{proof}


\begin{lemma}
\label{mgalebktk}
 For a positive supermartingale $X_t$

 \begin{equation*}
  P = P\left(\exists t \; s.t. \; X_t \geq B \mathbb{E}X_0 \right) \leq 1/B.
 \end{equation*}
\end{lemma}

\begin{proof}
Let $\tau$ be the first time that $X_t \geq B \mathbb{E}X_0$, and let $p_T = P\left(X_{\left(\tau \wedge T\right)} \geq B \mathbb{E}X_0  \right)$.  Then by optional stopping

 \begin{equation*}
  \mathbb{E}X_0 \geq E \left(X_{\tau \wedge T}\right) \geq E\left(X_{\tau \wedge T}; X_{\tau \wedge T} \geq B\mathbb{E}X_0\right) \geq p_T B \mathbb{E}X_0.
 \end{equation*}
 But $p_T \uparrow P$.

\end{proof}

\subsection{Proof of Theorem \ref{tmlowprob}}


\begin{proof}
We consider the functions $F_k$, $G_k$ in Lemma \ref{supMgale}.
To simplify notation, let $f_{k,\delta} = \left(F_{k-1} - \left(2-\delta/2\right)G_{k-1}\right)$, and note that
$$
|f_{k,\delta}| \le \frac{6\rho^3}{|\sin 2\theta|}=: \bar c/2.
$$
Lemma \ref{cuteTrick} tells us that $r_k^{-1} = Y_k$, and under the conditions of Lemma \ref{supMgale} (which are satisfied by assumption), the process
\begin{equation*}
e^{f_{k,\delta}\sigma^2 \left(1 - \delta\right)} \left(r_k e^{-\kappa k}\right)^{\left(\delta-1\right)} = \left(e^{f_{k,\delta}\sigma^2} Y_k e^{\kappa k}\right)^{\left(1-\delta\right)}
\end{equation*}
is a positive supermartingale. Now choose 

\begin{equation*}
\delta =  \frac{\kappa}{\sigma^2 \rho^2} + 224 \frac{c_0^3\rho\sigma}{|\sin 2\theta|}.
\end{equation*}
Then by our bound on $\kappa$ we have that

\begin{equation*}
 \delta \leq \frac{230c_0^3 \rho \sigma}{|\sin2\theta|}
\end{equation*}
and by our bound on $\sigma$ we have that 

\begin{equation*}
 \delta \leq 1/2
\end{equation*}
so that the conditions of Lemma \ref{mgalebktk} are satisfied. Then Lemma \ref{mgalebktk} gives

\begin{equation*}
P\left(\exists n: \left(e^{f_{n,\delta}\sigma^2 }Y_{n}e^{\kappa n}\right)^{1-\delta} \geq \left(e^{f_{1,\delta}\sigma^2}Y_{1}e^{\kappa}
e^{B - \bar{c}\sigma^2}\right)^{1-\delta} \right) \leq e^{-(B-\bar{c}\sigma^2)\left(1-\delta\right)}.
\end{equation*}
Taking logs, the event above is equivalent to
\begin{equation*}
\{\exists n: \log Y_{n} - \log Y_1 +\kappa \left(n-1\right)\geq B-\bar{c}\sigma^2 + \left(f_{1,\delta} - f_{n,\delta}\right)\sigma^2\}
\end{equation*}
which is a subevent of

\begin{equation*}
 \{\exists n:\log Y_{n} - \log Y_1 +\kappa \left(n-1\right)\geq B\}.
\end{equation*}
So the probability that the process $\log Y_n + \kappa n$ has a backtrack of size $B$ starting from time $1$ is at most $e^{-(B-\bar{c}\sigma^2)(1-\delta)}$. But

\begin{equation*}
 e^{-(B-\bar c \sigma^2)(1-\delta)} \leq e^{-B(1-\delta)}e^{\bar c \sigma^2}
\end{equation*}
and the bound on $\sigma$ gives

\begin{equation*}
 \bar c\sigma^2 \leq \frac{12\rho^3\sigma^2}{|\sin 2\theta|} \leq \frac{12}{460^2}
\end{equation*}
so that

\begin{equation*}
 e^{\bar c \sigma^2} \leq 2.
\end{equation*}
Now by $\kappa \leq 6c_0^3 \rho^3 \sigma^3/ |\sin 2 \theta|$ and our choice of $\delta$, we have

\begin{align*}
e^{-B(1-\delta)} \leq e^{-B\left(1- \frac{230c_0^3}{2\sin\theta|\sin2\theta|}\sigma \right)}
\end{align*}
meaning the probability that the process $\log Y_n + \kappa n$ has a backtrack of size $B$ starting from time $1$ is at most
$2e^{-B(1-230c_0^3\sigma /2\sin\theta|\sin 2\theta|)}$.

\end{proof}

%
%
%
%
%
%
%
%
%
%
%
%
%
%
%
%
%

\section{Proof of Theorem \ref{mainTheorem}}
\label{sect5}
\begin{proof}

Assume that $\mathbb{P}$ has support bounded by $c_0$. 
Recall that $N_n$ is the number of eigenvalues of the operator $H_{\omega,n}$ in the interval $[\lambda_0, \lambda_0+\lambda]$.
Let $\lambda_0 \in \left(-2,0\right) \cup \left(0,2\right)$, $n \in \mathbb{N}$, $\lambda>0$ and let $\sigma \leq \frac{2\sin\theta |\sin 2\theta|}{460c_0^3}$, so it satisfies the conditions of Theorem \ref{tmlowprob}.

Further, let $M = \frac{\sqrt{5}\sigma c_0^2}{2\sin^2\theta} \leq 1/2$, $\epsilon = 1$ and $\beta = \sigma^3$. We may assume that $\lambda \leq \sigma^3$, because otherwise the bound is trivial.

Choose $\kappa = \epsilon\left(\lambda+\beta\right)/\sin\theta + 2M\beta$. Then

\begin{equation*}
 \kappa \leq (\sigma^3+\sigma^3)/\sin\theta + \sigma^3 \leq 3 \sigma^3 / \sin \theta \leq 6 c_0^3 \rho^3 \sigma^3/|\sin 2\theta|.
\end{equation*}
By our choices above, and by Theorem \ref{MBound}, the conditions of Theorem \ref{tmbktkblp} are satisfied. So by Theorem \ref{tmbktkblp} we have that 
\begin{multline*}
  N_n \leq 1 + \text{the number of backtracks of size at least } \log \left(\epsilon\beta/\lambda\right) \text{ of } \log Y_n + \kappa n \\
  \leq 1 + \sum_{k=1}^n \mathbb{1}\left(\log Y_n + \kappa n \text{ has a backtrack of size } \log\left(\epsilon\beta/\lambda\right) \text{ starting at } k\right).
 \end{multline*}
Taking expectations and dividing both sides by $n$ yields
 
\begin{equation*}
  \frac{1}{n} E N_n \leq \frac{1}{n}\left(1 + nP\left(\log Y_n + \kappa n \text{ has a backtrack of size } \log\left(\epsilon\beta/\lambda\right)\right)\right).
\end{equation*}  
Now set $B = \log\left( \epsilon\beta/\lambda \right)$. Applying Theorem \ref{tmlowprob} gives 
%

 \begin{align*}
 \label{IDSpreLim}
  \frac{1}{n} E N_n  &\leq \frac{1}{n} + 2e^{-B (1-230c_0^3\sigma/2\sin \theta |\sin 2 \theta|) }\\
  &= \frac{1}{n} + 2\left( \frac{\lambda}{\epsilon \beta}  \right)^{1-230c_0^3\sigma/2\sin\theta|\sin 2 \theta|}\\
  &\leq \frac1n +  \frac{2}{\sigma^3} \lambda^{1-230c_0^3\sigma/2\sin\theta|\sin 2 \theta|}.
 \end{align*} 
Taking the limit as $n \to \infty$ yields
  
   \begin{equation*}
    \mu\left(\lambda_0, \lambda_0+\lambda\right) \leq \frac{2}{\sigma^3} \lambda^{1-230c_0^3\sigma/2\sin\theta|\sin 2\theta|}.
   \end{equation*}
  Now we use that
  
  \begin{align*}
   |2\sin\theta \sin2\theta| &= |2(\cos \theta) 2 \sin^2\theta | \\
   &= |\lambda_0|(4-\lambda_0^2)/2 \\
   &= \frac12 \left|\lambda_0\right|\left|2-\left|\lambda_0\right| \right|\left|2+|\lambda_0|\right| \\
   &\geq \left|\lambda_0\right|\left|2-\left|\lambda_0\right|\right|
  \end{align*}
so we have

  \begin{equation*}
    \mu\left(\lambda_0, \lambda_0+\lambda\right) \leq \frac{2}{\sigma^3} \lambda^{1-230c_0^3\sigma/\left|\lambda_0\right|\left|2-\left|\lambda_0\right|\right|}.
  \end{equation*}
  And note that
  
  \begin{equation*}
   |\lambda_0||2-|\lambda_0|| \geq \min(|\lambda_0|,2-|\lambda_0|)
  \end{equation*}
  so for $\lambda_0$ in $(-2+\gamma,-\gamma)\cup(\gamma,2-\gamma)$, 
  
  \begin{equation*}
   |\lambda_0||2-|\lambda_0|| \geq \gamma
  \end{equation*}
  giving
  
  \begin{equation*}
    \mu\left(\lambda_0, \lambda_0+\lambda\right) \leq \frac{2}{\sigma^3} \lambda^{1-230c_0^3\sigma/\gamma}.
  \end{equation*}
  
  Now the condition on $\sigma$ gives
  
  \begin{align*}
   460c_0^3 \sigma &\leq 2\sin\theta |\sin 2\theta| \\
    &= 2\sin^\theta|2\cos\theta| \\
    &= |\lambda_0| 2\sin^2\theta \\
    &= |\lambda_0| \frac{4-\lambda_0^2}{2}
  \end{align*}
so it is equivalent to

\begin{equation}
\label{LambdaCondition}
 \lambda_0(4-\lambda_0^2) \geq 920 c_0^3 \sigma.
\end{equation}
If this condition is violated, we have that

\begin{equation*}
 \frac{920c_0^3\sigma}{\gamma} \geq \frac{920 c_0^3 \sigma}{|\lambda_0||2-|\lambda_0||} \geq  |2+|\lambda_0|| \geq 2
\end{equation*}
meaning

\begin{equation*}
1 - 460c_0^3\sigma/\gamma \leq 0.
\end{equation*}
This means that by allowing an extra factor of $2$ in the constant of the exponent of $\lambda$, the bound on the IDS is trivially satisfied for $\lambda_0$ violating \eqref{LambdaCondition}. In other words, if we loosen our bound on the IDS from

\begin{equation*}
 \mu\left(\lambda_0, \lambda_0+\lambda\right) \leq \frac{2}{\sigma^3} \lambda^{1-230c_0^3\sigma/\left|\lambda_0\right|\left|2-\left|\lambda_0\right|\right|}
\end{equation*}
to

\begin{equation*}
 \mu\left(\lambda_0, \lambda_0+\lambda\right) \leq \frac{2}{\sigma^3} \lambda^{1-460c_0^3\sigma/\left|\lambda_0\right|\left|2-\left|\lambda_0\right|\right|}
\end{equation*}
we may drop the condition on $\lambda_0$. This completes the proof.
 \end{proof}

\bigskip
\noindent\textbf{Acknowledgements}
\phantomsection\addcontentsline{toc}{section}{Acknowledgements}\quad We thank Andrew Stewart and Jeffrey Schenker for careful reading and useful comments on earlier versions. This research was supported by the Canada Research Chair program and the NSERC Discovery Accelerator Grant of the second author. 

\setstretch{1}
\phantomsection\addcontentsline{toc}{section}{References}
\bibliographystyle{dcu}
\bibliography{References}

\bigskip\bigskip\bigskip\noindent
\begin{minipage}{0.49\linewidth}
Eric Hart
\\Department of Mathematics
\\University of Toronto
\\Toronto ON~~M5S 2E4, Canada
\\{\tt eric@math.toronto.edu}
\end{minipage}
\begin{minipage}{0.49\linewidth}
B\'alint Vir\'ag
\\Departments of Mathematics and Statistics
\\University of Toronto
\\Toronto ON~~M5S 2E4, Canada
\\{\tt balint@math.toronto.edu}
\end{minipage}

\end{document}